\documentclass[12pt]{article}
\usepackage{amsfonts,url}

\newcommand{\TA}{\mathop{\mathrm{TA}}\nolimits}
\newcommand{\DA}{\mathop{\mathrm{DA}}\nolimits}
\newcommand{\SA}{\mathop{\mathrm{SA}}\nolimits}
\newcommand{\rank}{\mathop{\mathrm{rank}}\nolimits}
\newcommand{\rect}{\mathop{\mbox{\textrm{R}}}\nolimits}
\newcommand{\pg}{\mathop{\mathrm{PG}}\nolimits}
\newcommand{\pgl}{\mathop{\mathrm{PGL}}\nolimits}
\newcommand{\gf}{\mathop{\mathrm{GF}}\nolimits}
\newcommand{\gd}{\mathop{\mathrm{GD}}\nolimits}
\newcommand{\pgaml}{\mathop{\mathrm{P}\Gamma\mathrm{L}}\nolimits}
\newcommand{\biplane}{\Xi}

\newtheorem{theorem}{Theorem}[section]
\newtheorem{corollary}[theorem]{Corollary}
\newtheorem{proposition}[theorem]{Proposition}
\newtheorem{preex}{Example}
\newenvironment{ex}{\preex\rm}{\endpreex}
\newtheorem{unnumber}{}

\newtheorem{prepf}[unnumber]{Proof}
\newenvironment{proof}{\prepf\rm}{\endprepf}

\begin{document}

\title{Sesqui-arrays, a generalisation of triple arrays}
\author{R. A. Bailey\footnote{School of Mathematics and Statistics, University
of St Andrews, North Haugh, St Andrews, Fife KY16 9SS, U.K.},
Peter J. Cameron$^*$ and Tomas Nilson\footnote{Department of Science Education and Mathematics, Mid Sweden University, SE-851 70 Sundsvall, Sweden}}
\date{}

\maketitle

\begin{abstract}%
A triple array is a rectangular array containing letters, each letter occurring
equally often with no repeats in rows or columns, such that the number of 
letters common to two rows, two columns, or a row and a column are (possibly
different) non-zero constants. Deleting the condition on the letters common
to a row and a column gives a double array. We propose the term
\emph{sesqui-array} for such an array when only the condition on pairs of
columns is deleted. Thus all triple arrays are sesqui-arrays.

In this paper we give three constructions for sesqui-arrays. The first gives
$(n+1)\times n^2$ arrays on $n(n+1)$ letters for $n\geq 2$. (Such an array
for $n=2$ was found by Bagchi.) This construction uses  Latin squares.
The second uses the \emph{Sylvester graph}, a subgraph of the
Hoffman--Singleton graph, to build a good block design for $36$ treatments in
$42$ blocks of size~$6$, and then uses this in a $7\times 36$ sesqui-array for
$42$ letters.

We also give a construction for $K\times(K-1)(K-2)/2$ sesqui-arrays on 
$K(K-1)/2$ letters.
This construction uses biplanes. It
starts with a block of a biplane and produces an array which satisfies the
requirements for a sesqui-array except possibly that of having no
repeated letters in a row or column. We show that this condition holds if
and only if the \emph{Hussain chains} for the selected block contain no
$4$-cycles. A sufficient condition for the construction to give a triple
array is that each Hussain chain is a union of $3$-cycles; but this condition
is not necessary, and we give a few further examples.

We also discuss the question of which of these arrays provide good designs for
experiments.
\end{abstract}

\section*{Dedication}
This paper is dedicated to the memory of Anne Penfold Street.
Throughout her career, her research focussed on various combinatorial designs.
However, she also linked these ideas to experimental design, as the book
\cite{SS} shows.  Indeed, some of her combinatorial research on neighbour designs
\cite{grid2,grid3,grid1,ebibd}
was inspired by problems in the design of agricultural experiments.  
We hope that this paper also manages to bridge the two areas.

The second author had the privilege of being taught by Anne at the University
of Queensland in the 1960s.  In those far-off days, there was no such subject
as combinatorics or design theory; Anne's lectures were on measure theory.
The first author remembers with gratitude Anne's hospitality (both mathematical
and social) during various research visits to the University of Queensland.

\section{Introduction}
\label{sec:intro}
\subsection{Definitions and Notation}

Suppose that we have an $r\times c$ array~$\Delta$ in which each of the
$rc$~cells contains one letter from a set of size $v$.  In order to exclude
Latin squares and Youden squares, we assume that $v > \max\{r,c\}$.
Figures~\ref{fig:TA1}--\ref{fig:SA1} show three such arrays.

\begin{figure}
\[
\begin{array}{|c|c|c|c|c|c|}
\hline
A & F & C & D & H & J\\
\hline
B & A & I & J & E & H\\
\hline
C & H & G & B & I & D\\
\hline
D & G & A & I & F & E\\
\hline
E & B & J & F & C & G\\
\hline
\end{array}
\]
  \caption{A triple array with $r=5$, $c=6$ and $v=10$}
  \label{fig:TA1}
  \end{figure}

\begin{figure}
  \[
  \begin{array}{|c|c|c|c|}
    \hline
    A & B & C & D\\
    \hline
    F & A & B & E\\
    \hline
    C & D & E & F\\
    \hline
    \end{array}
  \]
  \caption{A double array with $r=3$, $c=4$ and $v=6$}
  \label{fig:DA1}
\end{figure}

\begin{figure}
  \[
  \begin{array}{|c|c|c|c|c|c|c|}
\hline
A & H & B & G & C & F \\
\hline
B & G & F & C & E & D\\
\hline
C & F & E & D & A & H\\
\hline
D & E & A & H & G & B\\
\hline
  \end{array}
  \]
  \caption{A sesqui-array with $r=4$, $c=6$ and $v=8$}
  \label{fig:SA1}
  \end{figure}

Such an array gives rise to six incidence matrices.
The incidence matrix $N_{LR}$ is the $v \times r$ matrix whose $(i,j)$-entry is
the number of times that letter~$i$ occurs in row~$j$;
while $N_{RL} = N_{LR}^\top$.
The matrices $N_{LC}$ and $N_{CL}$ are defined similarly.
Likewise, $N_{RC}$ is the $r\times c$ matrix whose $(i,j)$-entry is the number
of letters in the unique cell where row~$i$ meets column~$j$: the assumption
at the start of this section is
that every entry in $N_{RC}$ is equal to $1$.  Also $N_{CR} = N_{RC}^\top$.

The array $\Delta$ also defines various component designs.  In the row component
design $\Delta_{R(L)}$ we consider rows as points and letters as blocks. Thus
$\Delta_{R(L)}$ is balanced if $N_{RL}N_{LR}$ is completely symmetric, which means
that it is a linear combination of the identity matrix and the all-$1$ matrix.
The column
component design $\Delta_{C(L)}$ is analogous.  The duals of these designs are
$\Delta_{L(R)}$ and $\Delta_{L(C)}$.
The component design $\Delta_{L(R,C)}$ is really what we have described, by
considering which letter occurs in which cell. Its dual is $\Delta_{R,C(L)}$,
which is concerned with which row-column combinations are allocated to each
letter.  We shall return to these last two components in Section~\ref{sec:opt}.

Here are some conditions that the array $\Delta$ may satisfy, with
names often used in the statistics literature.
\begin{itemize}\itemsep0pt
\item[(A0)]
  No letter occurs more than once in any row or any column.
  (The component designs $\Delta_{R(L)}$ and $\Delta_{C(L)}$ are both
  \emph{binary}; equivalently, all the entries in $N_{LR}$ and $N_{LC}$ are in
  $\{0,1\}$.)
\item[(A1)]
  Each letter occurs a constant number $k$ of times, where $vk=rc$.
(The array is \emph{equireplicate}.)
\item[(A2)]
  The number of letters common to any two rows is a non-zero constant
  $\lambda_{rr}$. (The component design $\Delta_{R(L)}$
  is \emph{balanced}.)
\item[(A3)]
  The number of letters common to any two columns is a non-zero constant
  $\lambda_{cc}$. (The component design $\Delta_{C(L)}$
  is \emph{balanced}.)
\item[(A4)]
  The number of letters common to any row and any column is a constant
  $\lambda_{rc}$.
  (In the context that each cell contains exactly one letter, this means that
  rows and columns have \emph{adjusted orthogonality} with respect to letters.)
\end{itemize}

Condition (A0) implies that the diagonal entries of $N_{RL}N_{LR}$
are all equal to $c$; and those of $N_{CL}N_{LC}$ 
are all equal to $r$.
Conditions (A0) and (A1) imply that the diagonal entries of $N_{LR}N_{RL}$ and
$N_{LC}N_{CL}$ are all equal to $k$.

Condition (A3) states that the off-diagonal entries of $N_{CL}N_{LC}$ are all
equal.  In general, we shall call these entries the \emph{column-intersection
  numbers}.  In the context of the component design $\Delta_{C(L)}$, they are
usually called \emph{concurrences}.

Condition (A4) states that every entry of $N_{RL}N_{LC}$ is equal
to $\lambda_{rc}$.

In the literature of the 21st century, an array which has $v>\max\{r,c\}$
and which satisfies all five conditions (A0)--(A4) is called a
\emph{triple array}, while one satisfying conditions (A0)--(A3) is called
a \emph{double array}. We propose the term \emph{sesqui-array} for an array
satisfying (A0)--(A2) and (A4): the prefix ``sesqui'' means ``one-and-a-half'',
and, of the last three conditions, we require adjusted orthogonality and half
of the two balance conditions required for a triple array.
For these three types of array we use the notation
$\TA(v,k,\lambda_{rr},\lambda_{cc},\lambda_{rc}:r\times c)$,
$\DA(v,k,\lambda_{rr},\lambda_{cc}:r\times c)$ and
$\SA(v,k,\lambda_{rr},\Gamma,\lambda_{rc}:r\times c)$ respectively, where
$\Gamma$ denotes the set of intersection numbers for pairs of distinct columns. 
The array in Figure~\ref{fig:TA1}, given in \cite{pot}, is a $\TA(10,3,3,2,3: 5 \times 6)$;
the array in Figure~\ref{fig:DA1} is a $\DA(6,2,2,1:3\times 4)$;
while that in Figure~\ref{fig:SA1} is a $\SA(8,3,4,\{0,2\},3: 4 \times 6)$.

\subsection{An important inequality}
If $\Delta_{R(L)}$ is balanced then $N_{RL}$ has rank $r$.  This gives the
shortest proof of Fisher's Inequality, that balance implies $r\leq v$.
Adjusted orthogonality also leads to some useful inequalities, as we now show.

\begin{theorem}
  \label{th:weak}
  If there is exactly one letter in each cell of $\Delta$, and $\Delta$ has
  adjusted orthogonality, then
  $\rank(N_{RL}) + \rank(N_{LC}) \leq v+1$.
  \end{theorem}

\begin{proof}
  Condition (A4) implies that $N_{RL}N_{LC}$ has rank~$1$.
  The row-space of $N_{RL}$ has dimension $\rank(N_{RL})$, and this is mapped
  by $N_{LC}$ onto a space of dimension at most~$1$.  Hence
  $v - \rank(N_{LC}) \geq \rank(N_{RL}) -1$.
\end{proof}

\begin{corollary}
\label{cor:ses}
  If $\Delta$ is a sesqui-array then $v \geq r + \rank(N_{LC}) -1$.
  \end{corollary}

\begin{corollary}
  \label{cor:AO}
  If $\Delta$ is a triple array then 
\begin{equation}
  v \geq r+c -1.
  \label{eq:AO}
  \end{equation}
  \end{corollary}

Bagchi proved Theorem~\ref{th:weak} and Corollary~\ref{cor:AO} 
in \cite{Bagchi1998}.
Corollary~\ref{cor:AO} was also proved in \cite{extremal,triple}. 

The triple array in Figure~\ref{fig:TA1} satisfies inequality~(\ref{eq:AO}).
The double array in Figure~\ref{fig:DA1} does not satisfy condition (A4) 
but it does satisfy inequality~(\ref{eq:AO}), as do all known double arrays.
The sesqui-array in Figure~\ref{fig:SA1} satisfies the inequality in
Corollary~\ref{cor:ses} but not (\ref{eq:AO}) because $v-r+1 =5
> 4 = \rank(N_{LC})$
but $c=6$.

There are various statistical optimality criteria for designs, which we
discuss in Section~\ref{sec:opt}. If $r<v$ and the component design
$\Delta_{R(L)}$ is good by these criteria, then it is often true that $N_{RL}$
has rank~$r$.  Likewise, if $c<v$ and $\Delta_{C(L)}$ has good statistical
properties then usually $N_{LC}$ has rank~$c$.  Thus arrays which satisfy
conditions (A0), (A1) and (A4) and  which are useful in experimental design 
often satisfy inequality~(\ref{eq:AO}) even if they do not satisfy both of (A2)
and (A3). Most of the sesqui-arrays in this paper satisfy
inequality~(\ref{eq:AO}).

However, Corollary~\ref{cor:ses} does exclude some block designs from being
the column component of a sesqui-array.

\begin{ex}
\label{eg:1}
  Consider a block design for six points in eight blocks of size three.
  The average concurrence is $8/5$.  Up to isomorphism, the only block
  design with all concurrences in $\{1,2\}$ is the one made by
  developing the blocks $\{1,2,5\}$ and $\{1,3,5\}$ modulo~$6$.
  The incidence matrix of this block design has rank $6$, and so
  Corollary~\ref{cor:ses} shows that it is impossible to have a $4 \times 6$
  sesqui-array with this as its column component.
  \end{ex}

\subsection{The construction problem}

There are already some sesqui-arrays in the literature, but without this name.
Bagchi~\cite{Bagchi1996} constructed an infinite family of $(n+1)\times 2n$
sesqui-arrays with $n^2+n$ letters, partly by using mutually orthogonal Latin
squares. 
The set of column-intersection numbers is $\{0,1,2\}$ and
component $\Delta_{C(L)}$ is partially
balanced with respect to the rectangular association scheme  $\rect(2,n)$.
(We refer the reader to \cite[Chapter 11]{SS} and \cite{RAB:AS} for information
about association schemes and partially balanced designs.)
Figure~\ref{fig:SA5} gives an example with $n=4$.
Bagchi and van Berkum \cite{BB}
gave an infinite family of arrays which, when transposed,
are sesqui-arrays with $r=s$, $c=st$ and $v=s^2$,
where $s$ is a prime power and $t$ is the size of a difference set in $\gf(s)$.
The component $\Delta_{L(C)}$ is a square lattice design (see Section~\ref{sec:latt}). 
Some of these were also given by Eccleston and Street in \cite{EccSt}.

\begin{figure}
\[
\begin{array}{|c|c|c|c|c|c|c|c|}
\hline
A & F & G & H & E & P & R & K\\
\hline
I & B & K & L & S & J & H & M\\
\hline
M & N & C & P & L & Q & O & F\\
\hline
Q & R & S & D & N & G & I & T\\
\hline
E & J & O & T & A & B & C & D\\
\hline
\end{array}
\]
\caption{A sesqui-array with $r=5$, $c=8$ and $v=20$}
\label{fig:SA5}
\end{figure}

These papers all used one or more  direct constructions, 
in the sense of finding a rule specifying the letter
allocated to the cell in row~$i$ and column~$j$ and then proving that the rule produces arrays
with the properties desired.

Nilson and \"Ohman \cite{NO} constructed double arrays from projective planes;
Nilson and Cameron \cite{CN} constructed them from difference sets in finite
groups.

Some triple arrays have also been given with a direct construction.  Preece used cyclic 
constructions in \cite{DAPbcs,DAPbka,DAPOz}; 
Seberry \cite{dodgy}, Street \cite{DS}, Bagchi \cite{Bagchi1998} and 
Preece, Wallis and Yucas \cite{paley} used finite fields;
Bailey and Heidtmann \cite{extremal} used properties of the groups $A_5$ and $S_6$;
Yucas \cite{Yuc} used projective geometry;
Nilson and Cameron \cite{CN} used difference sets with multiplier $-1$ in finite
Abelian groups.

A different approach, for both sesqui- and triple arrays, is to specify the component designs
$\Delta_{L(R)}$ and $\Delta_{L(C)}$ in such a way that the desired conditions are satisfied.
Thus $\Delta_{L(R)}$ is specified by the set $R(i)$ of letters in row~$i$, for $i=1$, \ldots, $r$,
and $\Delta_{L(C)}$ is specified by the set $C(j)$ of letters in column~$j$, for $j=1$, \ldots, 
$c$.  Condition (A4) states that $\left|R(i) \cap C(j)\right| = \lambda_{rc}$ for all $i$ and $j$.
Given these sets, can we put one letter of $R(i) \cap C(j)$ into cell $(i,j)$, for all $i$ and $j$,
in such a way that the set of letters in row $i$ is $R(i)$ and the set of letters in
column $j$ is $C(j)$?

Agrawal took this approach in \cite{Agrawal1} for the extremal case that
$v=r+c-1$.  He observed:

\begin{proposition}
\label{agrawal_obs} The existence of an $r\times c$ triple array with $v$
letters, where $v=r+c-1$, implies the existence of a symmetric
balanced incomplete-block design for $v+1$ points in blocks of size~$r$.
\label{p:agrawal}
\end{proposition}

(A proof was subsequently given in \cite{triple}.)  Agrawal used this to 
give canonical sets $R(i)$ and $C(j)$ from the same set of $v$ letters.
For all the cases that he tried with $k>2$, he was able to choose a
representative element of $R(i) \cap C(j)$ to go in cell $(i,j)$ in such a way
that the representatives in each row were all distinct, as were the
representatives in each column, so that the outcome was a triple array.  He
conjectured that this is always possible if $k>2$.  So far, this conjecture
has been neither proved nor refuted. 

There are two further issues with Agrawal's approach.  The first is that,
unless $v=r+c-1$,
there may not be a canonical way of
labelling the blocks of $\Delta_{R(L)}$ and $\Delta_{C(L)}$ by the same set of
letters.  Thus it may be possible to permute the names of the letters in one
component while still satisfying (A4).  Is it possible that an acceptable set
of distinct representatives can be chosen for one labelling but not the other?

The second issue is that the general problem of finding an array of distinct
representatives, given the sets $R(i)$ and $C(j)$, has been shown to be
NP-complete by Fon-Der-Flaass  in \cite{vanderF}.

A third approach is exhaustive computer search.  Some of the arrays given in 
\cite{tripweb,triple,SW} were found like this.

In the next three sections of this paper we give some direct constructions of
sesqui-arrays, some of which turn out to be triple arrays.  On the way, we
find an incomplete-block design which seems to be a good practical substitute
for the non-existent affine plane of order~$6$. The following section assesses
whether these arrays have statistical properties desirable for experimental
designs, while the final section poses some suggestions for further work.

\section{Sesqui-arrays 
  from Latin squares}
\label{sec:LS}

Here we give a method of constructing a sesqui-array for $n(n+1)$ letters in
a rectangle with $n+1$ rows and $n^2$ columns. It works for every integer~$n$
with $n \geq 2$.
The array has $v=n^2+n = c+r-1$ and so it satisfies inequality~(\ref{eq:AO}).
As usual, we interpret ``letter'' to mean any symbol.

The method has three ingredients.  One is a Latin square $\Phi_1$ of order $n$
on a set $\Lambda_1$ of letters. The second is an $n \times n$ array $\Phi_2$
with $n^2$ distinct letters allocated to its cells; these letters make a
set $\Lambda_2$ disjoint from $\Lambda_1$.  The third is a Latin square $\Phi_3$
of order $n+1$ with letters $1$, $2$, \ldots, $n$ and $\infty$.

Here is the algorithm for constructing an $(n+1) \times n^2$ array $\Delta$
whose set of letters is $\Lambda_1 \cup \Lambda_2$.
\begin{enumerate}
\item
  Start with the array $\Phi_3$.
\item
  Remove the column in which $\infty$ occurs in the last row.
\item
  For $i=1$, \ldots, $n$, replace the occurrence of $\infty$ in row~$i$
  by the $i$-th row of $\Phi_1$.
  \item
  For $i=1$, \ldots, $n$, replace every  occurrence of the symbol~$i$
  by the $i$-th row of $\Phi_2$.
\end{enumerate}

Now, the final row of $\Delta$ contains every letter of $\Lambda_2$, while
row~$i$ replaces those in row~$j$ of $\Phi_2$ by $\Lambda_1$, where $j$~is
the symbol in $\Phi_3$ which is in row~$i$ of the column which has been
removed.  Hence every pair of rows have $n(n-1)$ letters in common.

Moreover, each column of $\Delta$ contains an entire column of $\Phi_2$
and one letter of $\Lambda_1$.  Thus it has $n$~letters in common with
every row.

Finally, consider any column of $\Delta$.  It has $n$~letters of~$\Lambda_2$
in common with every other column derived from the same column of $\Phi_2$.
It has one letter of $\Lambda_1$ in common with each of $n-1$ other columns,
and no letters in common with any other column.

Hence $\Delta$ is a
$\SA\left(n(n+1),n,n(n-1), \{0,1,n\},n: (n+1) \times n^2\right)$.

In fact, the columns of $\Delta$ can be labelled by ordered pairs whose
first elements are columns of $\Phi_2$ and whose second elements are letters
in $\Lambda_1$.  This labelling gives another $n \times n$ array, which gives
the rectangular association scheme $\rect(n,n)$ on the set of columns.

We have therefore proved the following.

\begin{theorem}
  If $n$ is an integer with $n\geq 2$ then there is an $(n+1) \times n^2$
  sesqui-array with $n(n+1)$ letters whose column component design is
  partially balanced with respect to the rectangular association scheme
  $\rect(n,n)$.  Each column has intersection number $0$ with $(n-1)^2$
  other columns, $1$ with $n-1$ other columns, and $n$ with $n-1$ other
  columns.
\end{theorem}

\begin{ex}
  When $n=2$ we may take
  \[
  \Phi_1 = \begin{array}{|c|c|}
    \hline
    A & B\\
    \hline
    B & A\\
    \hline
  \end{array}\ ,
  \quad
  \Phi_2 = \begin{array}{|c|c|}
    \hline
    C & D\\
    \hline
    E & F\\
    \hline
  \end{array}
  \quad \mbox{and} \quad
  \Phi_3 = \begin{array}{|c|c|c|}
    \hline
    1 & 2 & \infty\\
    \hline
    2 & \infty & 1\\
    \hline
    \infty & 1 & 2\\
    \hline
    \end{array}\ .
  \]
  Thus $\Lambda_1 = \{A,B\}$ and $\Lambda_2 = \{C,D,E,F\}$.
  Removing the first column of $\Phi_3$, replacing each $\infty$ by the
  appropriate row of $\Phi_1$, replacing each $1$ by $(C,D)$ and each $2$
  by $(E,F)$ gives the sesqui-array in Figure~\ref{fig:SA2}.
  Up to relabelling of the letters, this is identical to an example
  given by Bagchi in \cite{Bagchi1996}.  

  \begin{figure}
    \[
    \begin{array}{|c|c||c|c|}
      \hline
      E & F & A & B\\
      \hline
      B & A & C & D\\
      \hline
      C & D & E & F\\
      \hline
      \end{array}
    \]
    \caption{A sesqui-array with $r=3$, $c=4$ and $v=6$:
      double vertical lines indicate the method of construction}
    \label{fig:SA2}
    \end{figure}
  \end{ex}

\begin{ex}
  For $n=4$, put
  \[
  \Phi_1 =
  \begin{array}{|c|c|c|c|}
    \hline
    A & B & C & D\\
    \hline
    D & A & B & C\\
    \hline
    C & D & A & B\\
    \hline
    B & C & D & A\\
    \hline
    \end{array}\ ,
  \quad
  \Phi_2 = \begin{array}{|c|c|c|c|}
    \hline
    E & F & G & H\\
    \hline
    I & J & K & L\\
    \hline
    M & N & O & P\\
    \hline
    Q & R & S & T\\
    \hline
  \end{array}
\quad \mbox{and} \]
\[
\Phi_3 = \begin{array}{|c|c|c|c|c|}
  \hline
  \infty & 1 & 2 & 3 & 4\\
  \hline
  4 & \infty & 1 & 2 & 3\\
  \hline
  3 & 4 & \infty & 1 & 2\\
  \hline
  2 & 3 & 4 & \infty & 1\\
  \hline
  1 & 2 & 3 & 4 & \infty\\
  \hline
  \end{array}\ .
\]
Removing the last column of $\Phi_3$ gives the sesqui-array in
Figure~\ref{fig:SA3}.

\begin{figure}
  \[
  \begin{array}{|c|c|c|c||c|c|c|c||c|c|c|c||c|c|c|c|}
    \hline
    A & B & C & D & E & F & G & H & I & J & K & L & M & N & O & P\\
    \hline
    Q & R & S & T & D & A & B & C & E & F & G & H & I & J & K  & L\\
    \hline
    M & N & O & P & Q & R & S & T & C & D & A & B & E & F & G & H\\
    \hline
    I & J & K & L & M & N & O & P & Q & R & S & T & B & C & D & A\\
    \hline
    E & F & G & H & I & J & K & L & M & N & O & P & Q & R & S & T\\
    \hline
    \end{array}
  \]
  \caption{A sesqui-array with $r=5$, $c=16$ and $v=20$:
    double vertical lines indicate the method of construction}
  \label{fig:SA3}
  \end{figure}
\end{ex}

This construction has the advantage of being very flexible, because it
can be used for all 
integers $n$ greater than $1$ and because there are
no constraints on the isotopism classes of the Latin squares $\Phi_1$ and
$\Phi_3$.
For most values of $n$, there are no triple arrays with these parameters.
However, the column component $\Delta_{C(L)}$ has very unequal
concurrences, particularly as $n$ increases, and this makes the design
inefficient in the sense discussed in Section~\ref{sec:opt}.

When $n=2$, there is no triple array for this parameter set.  That is why the arrays in
Figures~\ref{fig:DA1} and~\ref{fig:SA2} are not isomorphic.  The first satisfies condition~(A3)
and the second satisfies (A4) but neither satisfies both.

For $n\in \{3,4,5\}$, Sterling and Wormald gave triple arrays for these
parameter sets in \cite{SW}.  Figure~\ref{fig:TA2} shows that for $n=3$.
Agrawal also gave triple arrays for these in \cite{Agrawal1}.
McSorley at al.\ gave one for
$n=7$ in \cite{triple}, and also constructed some for $n \in \{8,11,13\}$, which
can be found at \cite{tripweb}.

\begin{figure}
\[
\begin{array}{|c|c|c|c|c|c|c|c|c|}
\hline
D & H & F & L & E & K & I & G & J\\
\hline
A & K & I & B & J & G & C & L & H\\
\hline
J & A & L & D & B & F & K & E & C\\
\hline
G & E & A & H & I & B & D & C & F\\
\hline
\end{array}
\]
\caption{A triple array with $r=4$, $c=9$ and $v=12$}
\label{fig:TA2}
\end{figure}

For $n=6$ the column component cannot be balanced, as there is no affine
plane of order~$6$.  The next section gives an efficient new design
for $36$ points in $42$ blocks of size $6$, and then a sesqui-array
which has this as its column component.

\section{A design for the case $n=6$}
\label{sec:HS}

\subsection{Square lattice designs}
\label{sec:latt}
If $n$ is a power of a prime then there is an affine plane of order~$n$.
It has $n^2$ points, and its $n(n+1)$ lines form the blocks of a balanced
incomplete-block design. If $2 \leq r \leq n+1$, then the blocks of any
$r$ parallel classes give a block design known as a
\emph{square lattice design} \cite{FY} or \emph{net} \cite{Brnet}.
Even if there is no affine plane of order~$n$, if there
are $r-2$ mutually orthogonal Latin squares of order~$n$ then there is a square
lattice design for $n^2$ points in $rn$ blocks of size~$n$. Such designs are
known to be optimal in the sense discussed in Section~\ref{sec:opt}.  All
concurrences are in $\{0,1\}$.

When $n=6$ then there is no pair of orthogonal Latin squares and so there
are square lattice designs for $r=2$ and $r=3$ but for no larger values of~$r$.
This lack does not prevent the need for efficient block designs for $36$ points
in $6r$ blocks of size $6$ for larger values of $r$. A heuristic for a good
block design (in the sense explained in Section~\ref{sec:opt}) is that all
concurrences are in $\{0,1,2\}$, since they cannot all be
equal and their average value is $r/(n+1)$, which is at most~$1$.

Patterson and Williams gave a good block design for $n=6$ and $r=4$ in
\cite{PW}; it has all concurrences in $\{0,1,2\}$.
We are not aware of similar block designs for $r=5$ or $r=6$.
In the following subsections we construct such block designs for
$r\in\{4, 5,6,7\}$ and then construct a sesqui-array which has the one with
$r=7$
as its column component.

\subsection{Construction of the block designs}
\label{sec:BD36}
Our construction uses a property of the number $6$:
the symmetric group $S_6$ admits an outer
automorphism, and $6$ is the only cardinal number for which this is true.
See \cite[Chapter 6]{cvl} for the use of this outer automorphism in various
constructions, including the \emph{Hoffman--Singleton graph} \cite{hoffsing}:
this is a  graph on $50$ vertices with valency $7$, diameter $2$ and girth $5$.
We require no knowledge of the outer automorphism of $S_6$, but write the 
construction just in terms of the Hoffman--Singleton graph $HS$, and its subgraph
the Sylvester graph.
Details of these graphs can be found at~\cite{drgpage}.

Let $a_0$ and $b_0$ be adjacent vertices of the graph $HS$; let $A$ be the set of
neighbours of $a_0$ excluding $b_0$, and $B$ the set of neighbours of $b_0$
excluding $a_0$. Since there are no triangles, $A\cap B=\emptyset$.
This accounts for $2+6+6=14$ of the $50$ vertices.
If $x$ is one of the remaining $36$ vertices, then $x$ lies at distance $2$
from both $a_0$ and $b_0$, and hence $x$ is joined
to a unique vertex $a$ in $A$ and a unique vertex in $b$ in $B$.
(Uniqueness holds because there are no quadrilaterals in the graph $HS$.)
So we can label $x$ with the ordered pair
$(a,b)$, and identify the set of vertices non-adjacent to $a_0$ and $b_0$
with the Cartesian product $A\times B$.

We are particularly concerned with the induced subgraph of the
graph $HS$ on the set $A\times B$. This graph is known as
the \emph{Sylvester graph} \cite[Theorem~13.1.2]{bcn}. We denote it by $\Sigma$.
Each vertex $(a,b)$ is joined in $HS$ to $a$, $b$, and five vertices
in $A\times B$. Of the six vertices in $A\times B$ in the closed neighbourhood
of $(a,b)$, no two have the same first or second coordinate (since this would
create a short cycle in $HS$).

The elements of $A \times B$ form the points of our new design $\Theta$.
The $42$ blocks are labelled by the elements of $B \cup (A \times B)$.
The block labelled $b$ contains the six points $(a,b)$ for $a$ in $A$;
the block labelled $(a,b)$ contains the point $(a,b)$ and all its
neighbours in the graph $\Sigma$.  Thus every block contains six points
and the design is binary.

Now we check the concurrences of pairs of distinct points $(a_1,b_1)$ and
$(a_2,b_2)$.  If $b_1=b_2$ then the unique block in which they concur is
block $b_1$.  If $a_1=a_2$ then no block contains both.
If points $(a_1,b_1)$ and $(a_2,b_2)$ are adjacent in $\Sigma$ then they
concur in blocks $(a_1,b_1)$ and $(a_2,b_2)$ but no others.
Otherwise, there is a unique point $(a_3,b_3)$ adjacent to both in $\Sigma$,
and so they concur in block $(a_3,b_3)$ but no other.
Hence the set of concurrences is $\{0,1,2\}$.

Note that this design $\Theta$
is partially balanced with respect to the four-class association
scheme whose relations are ``same first component'', ``same second component'',
``adjacent in $\Sigma$'', and ``none of the foregoing''.

The design is $\Theta$ also resolvable.  
The blocks labelled by the elements of $B$ form one replicate.
There are two ways of forming the remaining replicates: either group together all blocks
labelled $(a,b)$ for a fixed $a$, or do the same for a fixed $b$.  Removing one or more
replicates gives a resolvable design with $r<7$.

\subsection{Construction of the sesqui-array}

Now we construct the $7 \times 36$ array $\Delta$, whose column design is
$\Theta$. Our strategy is as follows: first we build an array $\Delta_0$ with
the correct column design; however, $\Delta_0$ fails spectacularly to be a
sesqui-array, having many repeats of letters in rows. Then we fix the problem
by permuting letters in columns to obtain $\Delta$.

The rows of $\Delta_0$ are labelled
with elements of $\{*\}\cup A$, where $*$ is a new symbol,
and the columns by $A\times B$, identified with the points of $\Theta$.
The letters are labelled by $B\cup(A\times B)$, identified with the blocks
of $\Theta$.
The rule for placing letters
in the array is as follows:
\begin{itemize}\itemsep0pt
\item The letter labelled $(a,b)$ goes in row $a$, column $(a',b')$, for
every $(a',b')$ adjacent to $(a,b)$ in the graph $\Sigma$ ($5$ occurrences),
and also in row $*$, column $(a,b)$ ($1$ occurrence).
\item The letter labelled $b$ goes in row $a$, column $(a,b)$ for all $a\in A$
($6$ occurrences).
\end{itemize}
We see at once that the column component $\Delta_{C(L)}$ is the design
$\Theta$.  Hence columns are binary, the design is equireplicate, and the
set of column concurrences is $\{0,1,2\}$. However, for each $a$ in $A$,
row $a$ contains letters
from $B$ each once, and letters of the form $(a,b)$ each five times, for each
$b\in B$, and no letters $(a',b)$ for $a'\ne a$. The following modification,
which simply permutes entries in columns, does not change the column design,
which thus remains binary and equireplicate, so that (A1) and half of (A0) are
satisfied.

Consider the six permutations (where we regard the entries and labels as taken
from the set $A=\{1,\ldots,6\}$:
\begin{eqnarray*}
\sigma_1 &=& (1)(6,5,4,3,2)\\
\sigma_2 &=& (2)(5,6,4,1,3)\\
\sigma_3 &=& (3)(6,2,5,1,4)\\
\sigma_4 &=& (4)(2,3,6,1,5)\\
\sigma_5 &=& (5)(3,4,2,1,6)\\
\sigma_6 &=& (6)(4,5,3,1,2)
\end{eqnarray*}
(These correspond to the blocks of a neighbour-balanced design for six
treatments in six circular blocks of size five given in \cite{forest}.)
It is readily checked that,
for each ordered pair $(a_1,a_2)\in A\times A$, there is a unique permutation
in the set which maps $a_1$ to $a_2$: the set is \emph{sharply transitive}.

Now take the array $\Delta_0$. Consider column $(a,b)$, which contains the
letters $(a,b)$ (in row $*$), $b$ (in row $a$), and $(a',b')$ (in row $a'$),
where $(a',b')$ is joined to $(a,b)$ in the graph~$\Sigma$. We permute
the elements of this column, fixing the entry in row $*$, so that $(a',b')$
is placed in row $a'\sigma_a$. Since $\sigma_a$ fixes $a$, the entry in
row $a$ is not changed; moreover, the set of elements in column $(a,b)$ is
not changed. Perform this operation on every column. Let $\Delta$ be the
resulting array.

We claim first that, in the resulting array, the letter $(a,b)$ is never
contained in row $a$. For in column $(a,b)$, the entry in row $a$ is $b$,
while in column $(a',b')$, the pair with first element $a$ has been moved
to row $a\sigma_{a'}$, which is not equal to $a$.

Next we claim that, in $\Delta$, no letter is repeated in a row. For suppose
that letter $(a,b)$ occurs twice in row $a'$, say in columns $(a_1,b_1)$
and $(a_2,b_2)$. Then $a\sigma_{a_1}=a'=a\sigma_{a_2}$. By the sharp transitivity of
our set of permutations, this implies that $a_1=a_2$. Now the property of
the Sylvester graph (that two points with the same first coordinate have no 
common neighbour) implies that $b_1=b_2$. This contradicts the assumption
that the columns $(a_1,b_1)$ and $(a_2,b_2)$ are distinct.

This completes the verification of (A0).

In $\Delta$, the letters in row $a$ are those in $B$ together with all 
$(a',b')$ with $a'\ne a$ (while those in row $*$ are all ordered pairs in
$A\times B)$. Thus any two rows have $30$ common letters, so the row design
$\Delta_{R(L)}$ is balanced; that is, (A2) holds.

Finally, we check condition (A4).
\begin{itemize}\itemsep0pt
  \item Row $*$ and column $(a,b)$ have letters $(a,b)$ and its five neighbours
in common ($1+5=6$ of these).
\item Row $a$ and column $(a,b)$ have in common the
  letters $b$ and $(a',b')$ for the five
neighbours $(a',b')$ of the vertex $(a,b)$ in $\Sigma$ ($1+5=6$ of these).
\item Row $a$ and column $(a',b)$ with $a'\ne a$ have in common the letters
  $b$, $(a',b)$, and
  the four neighbours $(a'',b')$ of the vertex
  $(a',b)$ with $a''\ne a$ ($1+1+4=6$ of
these).
\end{itemize}

Thus $\Delta$ is a sesqui-array $\SA(42,6,30,\Gamma,6:7\times36)$,
where $\Gamma=\{0,1,2\}$.  It is shown (transposed) in Figure~\ref{fig:SA4}.
In the figure, we have labelled the elements of both $A$ and $B$ as
$1,\ldots,6$; the position in the array determines whether a digit belongs to
$A$ or $B$. (In an ordered pair, the first member is in $A$ and the second in
$B$; the column labels other than $*$
are elements of $A$, and the single entries in table
cells are in $B$.)

\begin{figure}[htbp]
\[
\renewcommand{\arraystretch}{0.9}
\begin{array}{c||c|c|c|c|c|c|c|}
&*&1&2&3&4&5&6\\
  \hline
  \hline
(1,1)&(1,1)&1&(3,6)&(4,5)&(5,4)&(6,3)&(2,2)\\
\hline
(1,2)&(1,2)&2&(3,4)&(4,3)&(5,5)&(6,6)&(2,1)\\
\hline
(1,3)&(1,3)&3&(3,5)&(4,2)&(5,6)&(6,1)&(2,4)\\
\hline
(1,4)&(1,4)&4&(3,2)&(4,6)&(5,1)&(6,5)&(2,3)\\
\hline
(1,5)&(1,5)&5&(3,3)&(4,1)&(5,2)&(6,4)&(2,6)\\
\hline
(1,6)&(1,6)&6&(3,1)&(4,4)&(5,3)&(6,2)&(2,5)\\
\hline
(2,1)&(2,1)&(4,4)&1&(1,2)&(6,5)&(3,3)&(5,6)\\
\hline
(2,2)&(2,2)&(4,6)&2&(1,1)&(6,4)&(3,5)&(5,3)\\
\hline
(2,3)&(2,3)&(4,5)&3&(1,4)&(6,6)&(3,1)&(5,2)\\
\hline
(2,4)&(2,4)&(4,1)&4&(1,3)&(6,2)&(3,6)&(5,5)\\
\hline
(2,5)&(2,5)&(4,3)&5&(1,6)&(6,1)&(3,2)&(5,4)\\
\hline
(2,6)&(2,6)&(4,2)&6&(1,5)&(6,3)&(3,4)&(5,1)\\
\hline
(3,1)&(3,1)&(5,5)&(6,4)&1&(1,6)&(2,3)&(4,2)\\
\hline
(3,2)&(3,2)&(5,6)&(6,3)&2&(1,4)&(2,5)&(4,1)\\
\hline
(3,3)&(3,3)&(5,4)&(6,2)&3&(1,5)&(2,1)&(4,6)\\
\hline
(3,4)&(3,4)&(5,3)&(6,1)&4&(1,2)&(2,6)&(4,5)\\
\hline
(3,5)&(3,5)&(5,1)&(6,6)&5&(1,3)&(2,2)&(4,4)\\
\hline
(3,6)&(3,6)&(5,2)&(6,5)&6&(1,1)&(2,4)&(4,3)\\
\hline
(4,1)&(4,1)&(6,6)&(5,3)&(2,4)&1&(1,5)&(3,2)\\
\hline
(4,2)&(4,2)&(6,5)&(5,4)&(2,6)&2&(1,3)&(3,1)\\
\hline
(4,3)&(4,3)&(6,4)&(5,1)&(2,5)&3&(1,2)&(3,6)\\
\hline
(4,4)&(4,4)&(6,3)&(5,2)&(2,1)&4&(1,6)&(3,5)\\
\hline
(4,5)&(4,5)&(6,2)&(5,6)&(2,3)&5&(1,1)&(3,4)\\
\hline
(4,6)&(4,6)&(6,1)&(5,5)&(2,2)&6&(1,4)&(3,3)\\
\hline
(5,1)&(5,1)&(2,6)&(4,3)&(6,2)&(3,5)&1&(1,4)\\
\hline
(5,2)&(5,2)&(2,3)&(4,4)&(6,1)&(3,6)&2&(1,5)\\
\hline
(5,3)&(5,3)&(2,2)&(4,1)&(6,5)&(3,4)&3&(1,6)\\
\hline
(5,4)&(5,4)&(2,5)&(4,2)&(6,6)&(3,3)&4&(1,1)\\
\hline
(5,5)&(5,5)&(2,4)&(4,6)&(6,3)&(3,1)&5&(1,2)\\
\hline
(5,6)&(5,6)&(2,1)&(4,5)&(6,4)&(3,2)&6&(1,3)\\
\hline
(6,1)&(6,1)&(3,4)&(1,3)&(5,2)&(2,5)&(4,6)&1\\
\hline
(6,2)&(6,2)&(3,3)&(1,6)&(5,1)&(2,4)&(4,5)&2\\
\hline
(6,3)&(6,3)&(3,2)&(1,1)&(5,5)&(2,6)&(4,4)&3\\
\hline
(6,4)&(6,4)&(3,1)&(1,5)&(5,6)&(2,2)&(4,3)&4\\
\hline
(6,5)&(6,5)&(3,6)&(1,4)&(5,3)&(2,1)&(4,2)&5\\
\hline
(6,6)&(6,6)&(3,5)&(1,2)&(5,4)&(2,3)&(4,1)&6\\
\hline
\end{array}
\]
\caption{\label{fig:SA4}A sesqui-array with $r=7$, $c=36$, $v=42$ (transposed)}
\end{figure}

\section{Sesqui-arrays 
  from biplanes}
\label{sec:biplane}

We now  construct some sesqui-arrays (and triple arrays) from
biplanes.

\subsection{Biplanes and Hussain chains}

Let $\biplane$ be a biplane (a symmetric $2$-$(V,K,2)$ design with
$V=1+{K\choose 2}$).
(Lower-case $v$ and $k$ are normally used, but these would conflict with the
notation used for our arrays.)
There are only finitely many biplanes known, with $K=3$, $4$, $5$, $6$, $9$, $11$, $13$; 
the numbers up to isomorphism are $1$, $1$, $1$, $3$, $4$, $5$, $2$ 
(there is no classification for $K=13$ as yet). 
All known biplanes are described in detail in Section 15.8
of Marshall Hall's \emph{Combinatorial Theory}, 2nd edition~\cite{hall}, the
standard reference on biplanes.

We caution the reader that there are several misprints in~\cite{hall}
in the details on the biplanes:
\begin{itemize}
\item
In the list of blocks for the biplane $B_1$ with $k=6$,
the entry $16$  in block $A_{10}$ should be $15$ (p.~322).
\item
In the description of the biplane $B_L(9)$, the second occurrence of $31$ in
the permutation $\psi$ on p.~324 should be $34$,
and the element $32$ in the second base block should be $22$ (pp.~324, 326).
\item
There are two misprints in the list of blocks of $B_H(9)$ on p.~325;
in the block beginning $2,8$, the entry $38$ should be $34$;
and in the block beginning $7,8$, the entry $10$ should be $16$. 
\item
In the table of chain-lengths on p.~333, the chain written as $5$-$5$-$3$
should be $5$-$3$-$3$.
\item
There are  two misprints in the generators of the automorphism group of $B(13)$
on p.~334: the generator $x$ should have $69$ inserted before $70$ in the last
cycle; and in $z$, the cycle $(15,25)$ should be $(15,24)$.
\end{itemize}
We do not claim to have spotted all the misprints. In addition we have
computed the Hussain chain lengths for the Aschbacher biplane and its dual
(see below); these are not given by Hall.

Let $B$ be a block of $\biplane$.
Any pair of points of $B$ lie in one further block of $\biplane$, and any other block~$B'$
of $\biplane$ not equal to~$B$ meets $B$ in two points.  So we can label the remaining
blocks by the $2$-subsets of $B$.

Following Hall, we define, for any point $q\notin B$, a graph of
valency $2$ on $B$ called a \emph{Hussain chain}. The edges of the graph
are the intersections with $B$ of the blocks containing $q$. Thus there are
$K$ edges, and any point of $B$ lies on two edges; so the graph is a union of
cycles. We call this graph $H(q)$. The lists of cycle lengths in Hussain
chains for all blocks of all the known biplanes with $K<13$ are given by Hall.

The collection of graphs $H(q)$ for $q\notin B$ has the following properties.
\begin{itemize}\itemsep0pt
\item[(H1)] 
Any two intersecting pairs of points of $B$ are both edges in $H(q)$ for a unique $q$.
\item[(H2)] 
Any two disjoint pairs of points of $B$ are both edges in $H(q)$ for exactly two values
of $q$.
\item[(H3)] Any two Hussain chains share two (disjoint) edges.
\end{itemize}
The proofs are straightforward. For every pair of points in $B$,
there is a unique block $B'$
meeting $B$ in just that pair; if two such blocks contain $q$, they can have
at most one point of $B$ in common.

The collection of Hussain chains determines the biplane:
\begin{itemize}\itemsep0pt
\item the points of $\biplane$ are the points of $B$ and the Hussain chains;
\item the blocks of $\biplane$ are a symbol $*$ and the pairs of points of $B$;
\item $*$ is incident with every point of $B$, every point of $B$ with every pair containing
it, and each pair with every Hussain chain of which it is an edge.
\end{itemize}

\subsection{The construction}

Let $\biplane$ be a biplane with $V$ points and block size $K$, and let $B$
be a block of $\biplane$.  We exclude the case $K=3$. We form an array~$\Delta$
of size $K\times(V-K)$, whose rows are indexed by the
points of $B$ and columns by the points outside $B$. The letters in the array
are indexed by the $2$-element subsets of $B$. In row $p$ and column $q$, we
put $\{p_1,p_2\}$, if $p_1$ and $p_2$ are the two neighbours of $p$ in the
Hussain chain $H(q)$; in other words, the blocks containing $p$ and $q$ meet
$B$ again in the points $p_1$ and $p_2$. 

We now check the conditions for a triple array.

\begin{itemize}\itemsep0pt
\item[(A0)] A letter cannot occur more than once in a row, by property (H1).
  In fact, the letters in row $p$ are the $2$-subsets of $B \setminus \{p\}$.
  On the other hand, $\{p_1,p_2\}$ could occur in column $q$
and rows $p$ and $p'$; indeed this happens if and only if $\{p,p_1,p',p_2\}$ is
a $4$-cycle in $H(q)$. We conclude that the array $\Delta$ is binary if and only
if no Hussain chain $H(q)$ contains a $4$-cycle.
\item[(A1)] For any letter $\{p_1,p_2\}$, for each $p\notin\{p_1,p_2\}$, there
  is a unique $q$ such that $H(q)$ has edges $\{p,p_1\}$ and $\{p,p_2\}$,
  by property~(H1). Thus
each letter occurs $K-2$ times in the array $\Delta$, which is therefore equireplicate.
\item[(A2)] Rows $p$ and $p'$ share all letters $\{p_1,p_2\}$ for which
$\{p_1,p_2\}\cap\{p,p'\}=\emptyset$. So any two rows have $K-2\choose2$ common
  letters. If $K>3$ this number is non-zero and so the
  component design $\Delta_{R(L)}$ is balanced.
\item[(A3)] It is not true in general that the component design $\Delta_{C(L)}$
is balanced. We return to this point later.
\item[(A4)] Each  row 
has $K-2$ common letters with each column. 
For if the row
and column indices are $p$ and $q$, then the common letters are the ``short
diagonals'' (pairs of vertices joined by a path of length $2$) of $H(q)$ which
do not contain $p$; there are $K-2$ of these. (The presence of $4$-cycles does
not affect this count, since the ``short diagonals'' of a $4$-cycle are
counted twice.) In other words, adjusted orthogonality always holds.
\end{itemize}

Thus, if we choose a biplane and a block such that no Hussain chain contains
a $4$-cycle, then we obtain a binary array $\Delta$ satisfying (A0)--(A2) and
(A4) but not necessarily (A3). As stated earlier, such an array is a
sesqui-array. So we have proved the first two parts of the following
result.

\begin{theorem}
Suppose that $B$ is a block of a biplane $\biplane$ with block size at least~$4$.
\begin{enumerate}
\item The array $\Delta$ constructed above satisfies (A1), (A2) and (A4).
\item It is binary if and only if none of the Hussain chains on
the block $B$ contains a $4$-cycle.
\item It is a triple array if every Hussain chain on $B$ is a union of
$3$-cycles.
\end{enumerate}
\end{theorem}

\begin{proof}
Only the third statement remains to be proved. So suppose that all the
Hussain chains on $B$ are unions of $3$-cycles. Then the array is binary.
If $\{p,p_1\}$ and $\{p,p_2\}$ are edges of $H(q)$, then so is
$\{p_1,p_2\}$: thus the symbols in column $q$ are the edges of $H(q)$.
Property (H3) shows that any two Hussain chains $H(q)$ and $H(q')$
share precisely two edges.
\end{proof}

There are just two known biplanes which have this property. The $(16,6,2)$
biplane $B_1$ in Hall~\cite{hall} has this property for any choice of block.
The $(37,9,2)$ biplane $B_H(9)$ has a unique block for which all the Hussain
chains consist of three triangles. These give a $6\times10$ triple array with
$15$ symbols, which is shown in this form in \cite[Figure~32]{BCC},
and a $9\times28$ triple array with $36$ symbols.

The tranpose of the latter is shown in Figure~\ref{fig:TA3}.
The nine points of $B$ are identified with the projective line $\pg(1,8)$.
The elements of $\gf(8)$ are denoted $0$, $1$, $a$, $b$, $c$, $d$, $e$ and $f$,
where $b=a^2$, $f=a^3=a+1$, $c=a^4=a+b$, $e=a^5=a+b+1$ and $d=a^6=b+1$.
The rows are labelled by the Hussain chains, and the letters by pairs of
points of $B$.  The group $\pgl(2,8)$ has $28$ subgroups of order $3$, each
having three orbits: these give the $28$ Hussain chains. 
The automorphism group of this array is $\pgaml(2,8)$, which has order $1512$.

\begin{figure}
  \[
  \addtolength{\arraycolsep}{-0.5\arraycolsep}
  \begin{array}{c@{}c@{}c||cc|cc|cc|cc|cc|cc|cc|cc|cc|}
    & & & \multicolumn{2}{c|}{0} & \multicolumn{2}{c|}{1} &
    \multicolumn{2}{c|}{a} & \multicolumn{2}{c|}{b} & \multicolumn{2}{c|}{f} &
    \multicolumn{2}{c|}{c} & \multicolumn{2}{c|}{e} & \multicolumn{2}{c|}{d}
    & \multicolumn{2}{c|}{\infty}\\
    \hline
    \hline
(0,1,\infty) & (a,c,b) & (d,e,f) & 1 & \infty & 0 & \infty & b & c & a & c
    & d & e & a & b & d & f & e & f & 0 & 1\\
    \hline
    (0,a,\infty) & (b,e,f) & (1,d,c) & a & \infty & c & d & 0 & \infty & e & f
    & b & e & 1 & d & b & f & 1 & c & 0 & a\\
    \hline
    (0,b,\infty) & (f,d,c) & (a,1,e) & b &\infty & a & e & 1 & e & 0 & \infty
    & c & d & d & f & 1 & a & c & f & 0 & b\\
    \hline
    (0,f,\infty) & (c,1,e) & (b,a,d) & f & \infty & c & e & b & d & a & d &
    0 & \infty & 1 & e & 1 & c & a & b & 0 & f\\
    \hline
    (0,c,\infty) & (e,a,d) & (f,b,1) & c & \infty & b & f & d & e & 1 & f & 1
    & b & 0 & \infty & a & d & a & e & 0 & c\\
    \hline
    (0,e,\infty) & (d,b,1) & (c,f,a) & e & \infty & b & d & c & f & 1 & d & a
    & c & a & f & 0 & \infty & 1 & b & 0 & e\\
    \hline
    (0,d,\infty) & (1,f,a) & (e,c,b) & d & \infty & a & f & 1 & f & c & e & 1
    & a & b & e & b & c & 0 & \infty & 0 & d\\
    \hline
    (a,f,\infty) & (0,b,c) & (e,d,1) & b & c & d & e & f & \infty & 0 & c & a
    & \infty & 0 & b & 1 & d & 1 & e & a & f\\
    \hline
    (b,c,\infty) & (0,f,e) & (d,1,a) & e & f & a & d & 1 & d & c & \infty & 0
    & e & b & \infty & 0 & f & 1 & a & b & c\\
    \hline
    (f,e,\infty) & (0,c,d) & (1,a,b) & c & d & a & b & 1 & b & 1 & a & e &
    \infty & 0 & d & f & \infty & 0 & c & e & f\\
    \hline
    (c,d,\infty) & (0,e,1) & (a,b,f) & 1 & e & 0 & e & b & f & a & f & a & b
    & d & \infty & 0 & 1 & c & \infty & c & d\\
    \hline
    (e,1,\infty) & (0,d,a) & (b,f,c) & a & d & e & \infty & 0 & d & c & f & b & c
    & b & f & 1 & \infty & 0 & a & 1 & e\\
    \hline
    (d,a,\infty) & (0,1,b) & (f,c,e) & 1 & b & 0 & b & d & \infty & 0 & 1 & c
    & e & e & f & c & f & a & \infty & a & d\\
    \hline
    (1,b,\infty) & (0,a,f) & (c,e,d) & a & f & b & \infty & 0 & f & 1 & \infty
    & 0 & a & d & e & c & d & c & e & 1 & b\\
    \hline
    (b,d,\infty) & (c,a,0) & (1,f,e) & a & c & e & f & 0 & c & d & \infty & 1
    & e & 0 & a & 1 & f & b & \infty & b & d\\
    \hline
    (f,1,\infty) & (e,b,0) & (a,c,d) & b & e & f & \infty & c & d & 0 & e & 1
    & \infty & a & d & 0 & b & a & c & 1 & f\\
    \hline
    (c,a,\infty) & (d,f,0) & (b,e,1) & d & f & b & e & c & \infty & 1 & e & 0
    & d & a & \infty & 1 & b & 0 & f & a & c\\
    \hline
    (e,b,\infty) & (1,c,0) & (f,d,a) & 1 & c & 0 & c & d & f & e & \infty &
    a & d & 0 & 1 & b & \infty & a & f & b & e\\
    \hline
    (d,f,\infty) & (a,e,0) & (c,1,b) & a & e & b & c & 0 & e & 1 & c & d &
    \infty & 1 & b & 0 & a & f & \infty & d & f\\
    \hline
    (1,c,\infty) & (b,d,0) & (e,a,f) & b & d & c & \infty & e & f & 0 & d &
    a & e & 1 & \infty & a & f & 0 & b & 1 & c\\
    \hline
    (a,e,\infty) & (f,1,0) & (d,b,c) & 1 & f & 0 & f & e & \infty & c & d &
    0 & 1 & b & d & a & \infty & b & c & a & e\\
    \hline
    (c,e,\infty) & (b,0,a) & (f,1,d) & a & b & d & f & 0 & b & 0 & a & 1 & d
    & e & \infty & c & \infty & 1 & f & c & e\\
    \hline
    (e,d,\infty) & (f,0,b) & (c,a,1) & b & f & a & c & 1 & c & 0 & f & 0 & b
    & 1 & a & d & \infty & e & \infty & d & e \\
    \hline
    (d,1,\infty) & (c,0,f) & (e,b,a) & c & f & d & \infty & b & e & a & e &
    0 & c & 0 & f & a & b & 1 & \infty & 1 & d\\
    \hline
    (1,a,\infty) & (e,0,c) & (d,f,b) & c & e & a & \infty & 1 & \infty & d
    & f & b & d & 0 & e & 0 & c & b & f & 1 & a\\
    \hline
    (a,b,\infty) & (d,0,e) & (1,c,f) & d & e & c & f & b & \infty & a & \infty
    & 1 & c & 1 & f & 0 & d & 0 & e & a & b\\
    \hline
    (b,f,\infty) & (1,0,d) &(a,e,c) & 1 & d & 0 & d & c & e & f & \infty & b &
    \infty & a & e & a & c & 0 & 1 & b & f\\
    \hline
    (f,c,\infty) & (a,0,1) & (b,d,e) & 1 & a & 0 & a & 0 & 1 & d & e & c &
    \infty & f & \infty & b & d & b & e & c & f\\
    \hline
  \end{array}
  \]
  \caption{Triple array with $r=28$, $c=9$ and $v=36$}
  \label{fig:TA3}
  \end{figure}

Even if the conditions are not all satisfied, interesting arrays can be
obtained. Consider the unique $(7,4,2)$ biplane. For any block $B$, the three
Hussain chains on $B$ are the three possible $4$-cycles on this set. Thus,
the array~$\Delta$  is not binary, but it is easy to see that it satisfies all
the other conditions for a triple array: every pair of points is the short
diagonal of a unique $4$-cycle, but occurs twice in the corresponding column.

Another example comes from the unique $(11,5,2)$ biplane. In
this case, all Hussain chains are pentagons, and they form one orbit of the
alternating group $A_5$ on pentagons. The other orbit consists of the pentagons
formed by the diagonals of those in the first orbit, which form another
biplane isomorphic to the first. So our construction, for which the pairs in
column $q$ are the diagonals in $H(q)$, can be represented by the edges of
the image of the biplane under an odd permutation of $B$, and so (just as in
the above proof) the columns are balanced. So we do indeed obtain a
triple array, which is displayed in this way in \cite[Figure 31]{BCC}.
This example shows that the condition in the third statement of
the theorem is sufficient but not necessary.

This phenomenon is even more widespread: all biplanes with $K=6$ give triple
arrays. There is a reason for this, and a surprising further fact:

\begin{proposition}
Let $\biplane$ be any biplane with $K=6$, and $B$ any block of $\biplane$.
Then the sesqui-array obtained from it by our construction is a triple array.
Moreover, applying Proposition~\ref{p:agrawal} to this triple array gives the
same biplane in all three cases (the one denoted by $B_1$ in Hall~\cite{hall}).
\end{proposition}

\begin{proof}
The letters in the array are $2$-subsets of the distinguished block $B$.
The column corresponding to a point $q\notin B$ contains the edges of the
``two-step graph'' of $H(q)$, that is, two points are joined if they have a
common neighbour in $H(q)$. We denote this graph by $H^*(q)$. We note that,
because $K=6$, each Hussain chain $H(q)$ is either a pair of triangles or a
hexagon, and so $H^*(q)$ is always a pair of triangles (and is equal to
$H(q)$ if $H(q)$ is a pair of triangles).

We show that all the graphs $H^*(q)$ are distinct; it follows, since there
are ten of them, that they must consist of all possible pairs of triangles.

Suppose that $H^*(q_1)=H^*(q_2)$. There are three cases:
\begin{itemize}\itemsep0pt
\item 
$H(q_1)$ and $H(q_2)$ are both double triangles. Then
\[H(q_1)=H^*(q_1)=H^*(q_2)=H(q_2),\]
so $q_1=q_2$.
\item 
$H(q_1)$ is a double triangle (so $H(q_1)=H^*(q_1)$) and $H(q_2)$ is
a hexagon. In this case, $H(q_1)=H^*(q_2)$ consists of the short diagonals
of the hexagon, and has no edges in common  with $H(q_2)$, contradicting
property (H3).
\item $H(q_1)$ and $H(q_2)$ are both hexagons, and they have the same
two-step graph. Assume that the vertices of $H(q_1)$ are $1,2,\ldots,6$ in
order round the hexagon. By (H3), without loss of generality, $H(q_2)$ contains the
edge $\{1,2\}$. Then the other neighbour of $2$ in $H(q_2)$ must be either
$3$ or $5$; the first is impossible since then $H(q_1)$ and $H(q_2)$ would
share adjacent edges $\{1,2\}$ and $\{2,3\}$, contradicting (H1). So $\{2,5\}$ is an edge of
$H(q_2)$. By the same argument, $\{1,4\}$ is also an edge. Now the only way
to complete this to a hexagon with the correct two-step graph is to have
also the edges $\{5,6\}$, $\{6,3\}$ and $\{3,4\}$. But then $H(q_1)$ and
$H(q_2)$ share three edges, a contradiction.
\end{itemize}

Now the letters in the array are $2$-subsets of $\{1,\ldots,6\}$. By
construction, the letters in row $p$ are all the $2$-sets not containing $p$;
and the letters in column $q$ are the edges of $H^*(q)$, which form a pair
of disjoint triangles, and every such pair occurs in some column. This is
exactly a description of the biplane $B_1$.
\end{proof}

Using the list of cycle lengths in the Hussain chains of the known biplanes
given by Hall, we see that we obtain (binary) sesqui-arrays from the biplanes
with block size $5$, $6$ (all three biplanes, and indeed all of these are
triple arrays), and $9$ ($B_H(9)$ and $B'_H(9)$,
for any block -- the first of these gives a triple array for one chosen
block, as shown in Figure~\ref{fig:TA3}, but we verified by computer
that none of the other blocks give triple arrays -- and $B_L(9)$, for $10$ of
the $37$ blocks). All of the biplanes with block size $11$ have $4$-cycles for
any choice of block.


Hall does not compute the chain lengths for the Aschbacher biplanes with block
size $13$, so we have done this computation; the results are as follows (using
Hall's notation for the cycle lengths in Hussain chains).

For $B(13)$, there are 
\begin{itemize}\itemsep0pt
\item one block with chain structure $10$-$3$ ($11$ points) and $13$ ($55$
points);
\item one block with chain structure $5$-$5$-$3$ ($11$ points) and $13$
($55$ points);
\item $22$ blocks with chain structure $7$-$3$-$3$ ($5$ points),  $5$-$5$-$3$
($5$ points), $10$-$3$ ($16$ points), $9$-$4$ ($5$ points), $13$ ($35$ points);
\item $55$ blocks with chain structure $7$-$3$-$3$ ($1$ point), $6$-$4$-$3$
($6$ points), $5$-$5$-$3$ ($1$ point), $10$-$3$ ($7$ points),
$9$-$4$ ($3$ points), $8$-$5$ ($6$ points), $7$-$6$ ($9$ points), 
$13$ ($33$ points).
\end{itemize}

For its dual $B'(13)$, there are
\begin{itemize}\itemsep0pt
\item two blocks with chain structure $6$-$4$-$3$ ($55$ points) and
$10$-$3$ ($11$ points);
\item $11$ blocks with chain structure $5$-$5$-$3$ ($1$ point), $8$-$5$
($10$ points), $7$-$6$ ($10$ points) and $13$ ($45$ points);
\item $11$ blocks with chain structure $7$-$3$-$3$ ($5$ points), $10$-$3$
($1$ point), $8$-$5$ ($10$ points), $7$-$6$ ($10$ points), and $13$
($40$ points);
\item $55$ blocks with chain structure $7$-$3$-$3$ ($2$ points), $6$-$4$-$3$
($4$ points), $5$-$5$-$3$ ($3$ points), $10$-$3$ ($13$ points),
$9$-$4$ ($5$ points), $8$-$5$ ($2$ points), $7$-$6$ ($5$ points), $13$
($32$ points).
\end{itemize}

Thus we obtain (binary) sesqui-arrays from the first two blocks of $B(13)$ and
from the $22$ blocks of the second and third types in $B'(13)$.



\section{Optimality}
\label{sec:opt}

\subsection{Optimality criteria}
The statistical quality of the block design $\Delta_{C(L)}$ is determined
  by the eigenvalues of its \textit{information matrix}, which is
  $I_c - (rk)^{-1}N_{CL}N_{LC}$, where $I_c$ is the identity matrix of order~$c$.
  The constant vectors have eigenvalue~$0$.  If this eigenvalue has multiplicity
  more than one then some differences between columns cannot be estimated and
  the design is said to be \textit{disconnected}.  Otherwise, let
  the remaining eigenvalues be $\mu_1$, \ldots, $\mu_{c-1}$, with
  $0<\mu_1 \leq \mu_2 \leq \cdots \leq \mu_{c-1}\leq 1$. 
   These eigenvalues are called \emph{canonical efficiency factors}.  For a
  good design they should all be as large as possible, but they are constrained
  by the equation
 \[
  \sum_{i=1}^{c-1} \mu_i = \frac{c(k-1)}{k}.
  \]

  The design is said to be A-optimal if it maximizes the harmonic mean $\mu_A$
  of $\mu_1$, \ldots, $\mu_{c-1}$; to be D-optimal if it maximizes the
  geometric mean $\mu_D$ of $\mu_1$, \ldots, $\mu_{c-1}$; and to be E-optimal if
  it maximizes  $\mu_1$. See \cite{BCC09,ShahSinha} and \cite[Section 1.7]{SS}.
 
  The canonical efficiency factors of the dual design $\Delta_{L(C)}$ are
  $\mu_1$, \ldots, $\mu_{c-1}$ and $v-c$ others equal to $1$. Thus $\Delta_{L(C)}$
is optimal, under any of the three criteria, if and only if $\Delta_{C(L)}$ is.

If $\Delta_{C(L)}$ is balanced then $\mu_1 = \cdots = \mu_{c-1} = c(k-1)/[(c-1)k]$.
This design is A-optimal, D-optimal and E-optimal.  When such a design exists,
no non-balanced design with these parameters is optimal on any of these
criteria.  If there is no balanced design available for given values of $c$,
$r$, $v$ and $k$, then we usually try to find a design whose optimality criteria
are not too far short of $c(k-1)/[(c-1)k]$.

For the particular case that $c=n^2$, $k=n$ and $v=rn$ with $r\leq n+1$, it is
known that a square lattice design, if it exists, is A-, D- and E-optimal:
see \cite{CSCRAB}.  It has $r(n-1)$ canonical efficiency factors equal to
$(r-1)/r$ and $(n+1-r)(n-1)$ equal to $1$.  Thus
\[
\mu_1 = \frac{r-1}{r} \quad \mbox{and} \quad \mu_A = \frac{rn-n+r-1}{rn-n+2r-1}.
\]

From now on, we concentrate on the criteria $\mu_A$ and $\mu_1$.

\subsection{Optimality properties of the column component designs}

In a sesqui-array, the row component design is balanced.  The column component
may not be, and so statistical properties of the whole array depend on it.
Here we give the efficiency factors for the non-balanced column components
used in this paper.

The column component of the sesqui-array in Figure~\ref{fig:SA1} is partially balanced with
respect to the group divisible association scheme $\gd(3,2)$. Its canonical efficiency factors
are $2/3$ (three times) and $1$ (twice), so that $\mu_1 = 2/3$ and $\mu_A = 10/13 \approx
0.769$.  The alternative column design mentioned in Example~\ref{eg:1} has
canonical efficiency factors $2/3$, $3/4$ (twice) and $11/12$ (twice), giving $\mu_1=2/3$ and
$\mu_A = 330/419 \approx 0.788$.  Thus these two designs are equally good on the 
E-criterion.  The second is slightly better on the A-criterion, but cannot be incorporated into
a sesqui-array.

Bagchi showed that the column components of the sesqui-arrays constructed 
in \cite{Bagchi1996}
are E-optimal when $n\geq 5$.  The column components of the sesqui-arrays 
in \cite{BB,EccSt} are
the duals of square lattice designs, and hence optimal on all three criteria.

For the sesqui-arrays constructed in Section~\ref{sec:LS}, 
the column component has canonical
efficiency factors $1/(n+1)$ and $n/(n+1)$ (both with multiplicity $n-1$) and $1$ 
(with multiplicity $(n-1)^2$).  Thus
\[
\mu_1 = \frac{1}{n+1} \quad \mbox{and} \quad \mu_A = \frac{n(n+1)}{2n^2 + n + 1}.
\]
On the other hand, a balanced square lattice design for these parameters, if it exists,
has $\mu_1 = \mu_A = n/(n+1)$.  Thus the column components of the sesqui-arrays in
Section~\ref{sec:LS} are far from optimal, and become worse as $n$ increases.

Now we consider the block design for $36$ points in $42$ blocks of size $6$ constructed 
in Section~\ref{sec:BD36}. Using properties of the association scheme described there, we can
show that the canonical efficiency factors are $11/14$, $6/7$, $19/21$ and $1$ 
with multiplicities $16$, $5$, $9$ and $5$ respectively.  These give $\mu_1 \approx 0.786$ 
and $\mu_A \approx 0.851$.  The non-achievable upper bounds for these given by the 
non-existent affine plane are both equal to $6/7$, which is approximately $0.857$, so this 
design seems to be very good, and may indeed be optimal.

For $r=4$, the unachievable upper bound on $\mu_A$ is $0.840$; the design given 
in \cite{PW} has $\mu_A \approx 0.836$.

\subsection{Optimality properties of the rectangular arrays}
Rectangular arrays are used for designed experiments in two different contexts.  In the first,
described by Preece \cite{DAPbcs,DAPbka} and others, the design is $\Delta_{R,C(L)}$.
There are $v$ blocks of size $k$, one set of $r$ treatments and another set of $c$ treatments.
One treatment from each set is applied to each unit in each block, and a response is measured
on that unit.  The aim is to estimate the effects of each set of treatments, under the 
assumption that there is no interaction, which means that they do not affect each other.
In order to remove the effects of any differences between blocks, the data have to be projected
onto the orthogonal complement of the $v$-dimensional space defined by the blocks. If $R$
and $C$ have adjusted orthogonality with respect to $L$ then no further adjustment is needed
in order to estimate the effects of $R$ and $C$.  Thus if there is an array $\Delta$
satisfying (A0), (A1)  and (A4) for which $\Delta_{R(L)}$ and $\Delta_{C(L)}$ are both
optimal then $\Delta$ is optimal for the combined experiment.

In the other use for designed experiments, the experimental units form an $r \times c$
rectangle and there are $v$ treatments, which must be allocated to those units.  Thus
the design is $\Delta_{L(R,C)}$.  The information matrix for letters in this design is
$I_v -(rk)^{-1}N_{LC}N_{CL} - (ck)^{-1}N_{LR}N_{RL} + (vk)^{-1}J_v$, where $J_v$ is the 
$v \times v$ matrix with all entries equal to $1$.  The optimality criteria for
$\Delta_{L(R,C)}$ are based on the non-trivial eigenvalues of this matrix.

Bagchi and Shah  proved in \cite{BagShah} that if  $\Delta$ is a triple array 
then $\Delta_{L(R,C)}$ is optimal with respect to all the standard optimality 
criteria among the class of equireplicate designs for $v$ treatments in a $r \times c$ 
rectangle.

The design $\Delta_{L(R,C)}$ is said to have \textit{general balance} 
in the sense of Nelder \cite{GB} if the
matrices $N_{LR}N_{RL}$ and $N_{LC}N_{CL}$ commute with each other.
Note that adjusted orthogonality implies general balance.

Shah and Sinha  state and prove the following as Theorem 4.4.1 of \cite{ShahSinha}.
 If the array $\Delta$ has adjusted orthogonality and each of the two component designs
$\Delta_{L(R)}$ and $\Delta_{L(C)}$
is optimal, then the whole design $\Delta_{L(R,C)}$ is optimal
among equireplicate designs with general balance.  Again, this applies to all the
standard optimality criteria.

Theorem 4.4.2 of \cite{ShahSinha} is the following.
If the array $\Delta$ is equireplicate and has adjusted orthogonality and one component
is E-optimal and the other component has E-criterion bigger than the first, then
the design $\Delta_{L(R,C)}$ is E-optimal 
(without restriction to equal replication or general balance).

Pages 81--82 of \cite{ShahSinha} give an example from John and Eccleston \cite{JAJEcc}  
with $r=4$, $c=6$ and $v=12$.  Two equireplicate arrays are compared.
One has adjusted orthogonality; the other does not even have general
balance but it is better for $\Delta_{L(R,C)}$ on the A-optimality criterion.

Denote by $\mu_{AR}$, $\mu_{AC}$ and $\mu_{ARC}$ the value of the A-criterion
for the component designs $\Delta_{L(R)}$, $\Delta_{L(C)}$ and $\Delta_{L(R,C)}$
respectively.
Eccleston and McGilchrist 
proved in \cite{EccMcG} that, for equireplicate designs,
\[
\frac{1}{\mu_{ARC}} \geq \frac{1}{\mu_{AR}} + \frac{1}{\mu_{AC}} - 1
\]
with equality if and only if the array has adjusted orthogonality.
This shows immediately that, among designs with adjusted orthogonality,
the best thing to do is, if possible, make sure that the component designs 
$\Delta_{L(R)}$ and $\Delta_{L(C)}$
are both
A-optimal.  This result led several authors to conjecture that an array with 
adjusted orthogonality in which each component design is A-optimal is either 
A-optimal overall or not far from A-optimal overall.  The counter-example above 
does not really destroy that.
However, Shah and Sinha 
pointed out that Eccleston and McGil\-christ's  proof 
assumes general balance. Hence the restriction in their own theorem.  

Nonetheless, it does appear that, to find an array that is good for either $\Delta_{R,C(L)}$
or $\Delta_{L(R,C)}$, a good strategy is to find a sesqui-array $\Delta$ 
whose column component
performs well on the relevant optimality criterion.

\section{Further directions}

We conclude with some open problems indicating further directions.

\begin{itemize}
\item
Our sesqui-arrays have $k=r-1$ (Sections~\ref{sec:LS}--\ref{sec:HS}) or $k=r-2$
(Section~\ref{sec:biplane}).
Those in \cite{Bagchi1996} have $k=2$; while those in \cite{BB} have $k$ equal to $r-1$,
$(r\pm 1)/2$, $(r+3)/4$, $(r-1)/4$ and $(r\pm\sqrt{r})/2$, for suitable values of $r$;
 and those in \cite{CN} have $r=2k\pm\sqrt{k}$.
What other values of $k$ are possible?
\item
We have seen sesqui-arrays that satisfy inequality~(\ref{eq:AO}) and others that do not.
Is there a weaker inequality, other than the one in Corollary~\ref{cor:ses}, that is satisfied by
all sesqui-arrays?
\item
Do all double arrays satisfy inequality (\ref{eq:AO})?
\item
Apart from Corollary~\ref{cor:ses}, what other constraints must the column
component design $\Delta_{C(L)}$ satisfy if it is to be incorporated into a sesqui-array?
\item
There are also questions about isomorphism of sesqui-arrays. Phillips, Preece
and Wallis~\cite{ppw} enumerated the $5\times6$ triple arrays: there are seven
isomorphism classes of these. Can such classifications be extended to
sesqui-arrays? More modestly, do non-isomorphic biplanes give rise (by our
construction) to non-isomorphic triple arrays? This is particularly interesting
for the $6\times10$ arrays with $15$ letters, where, as we have seen, they all
give rise to the same biplane.
\end{itemize}

\end{document}